\documentclass{amsart}
\usepackage{amsfonts,amssymb,amscd,amsmath,enumerate,verbatim,calc}
\usepackage[all]{xy}

\newcommand{\CM}{Cohen-Macaulay}

\newcommand{\wrt}{with respect to}

\newcommand{\I}{\mathbb{I} }
\newcommand{\n}{\mathfrak{n} }
\newcommand{\m}{\mathfrak{m} }

\newcommand{\Z}{\mathbb{Z} }
\newcommand{\W}{\mathbf{W}^{\bullet} }
\newcommand{\Y}{\mathbf{Y}_{\bullet} }
\newcommand{\Pb}{\mathbf{P}_{\bullet} }

\newcommand{\K}{\mathbb{K}_{\bullet} }
\newcommand{\Kc}{\mathbb{K}^{\bullet} }

\newcommand{\C}{\mathbf{C} }
\newcommand{\D}{\mathbf{D} }
\newcommand{\T}{\mathbf{T}^{\bullet}  }

\newcommand{\E}{\mathbf{E}^{\bullet}  }
\newcommand{\F}{\mathbf{F}_{\bullet}  }

\newcommand{\bx}{\mathbf{x}}
\newcommand{\bu}{\mathbf{u}}
\newcommand{\by}{\mathbf{y}}
\newcommand{\fb}{\mathbf{f}}
\newcommand{\rt}{\rightarrow}

\newcommand{\wh}{\widehat }

\newcommand{\image}{\operatorname{image}}
\newcommand{\coker}{\operatorname{coker}}
\newcommand{\cone}{\operatorname{cone}}
\newcommand{\grade}{\operatorname{grade}}
\newcommand{\depth}{\operatorname{depth}}
\newcommand{\projdim}{\operatorname{projdim}}
\newcommand{\injdim}{\operatorname{injdim}}

\newcommand{\ann}{\operatorname{ann}}
\newcommand{\cx}{\operatorname{cx}}
\newcommand{\Proj}{\operatorname{Proj}}

\newcommand{\cmd}{\operatorname{cmd}}
\newcommand{\Inj}{\operatorname{Inj}}
\newcommand{\Tot}{\operatorname{Tot}}
\newcommand{\embdim}{\operatorname{embdim}}

\newcommand{\rank}{\operatorname{rank}}
\newcommand{\Syz}{\operatorname{Syz}}

\newcommand{\Hom}{\operatorname{Hom}}
\newcommand{\Ext}{\operatorname{Ext}}
\newcommand{\Tor}{\operatorname{Tor}}

\theoremstyle{plain}

\newtheorem{theorem}{Theorem}[section]
\newtheorem{corollary}[theorem]{Corollary}
\newtheorem{lemma}[theorem]{Lemma}
\newtheorem{proposition}[theorem]{Proposition}

\theoremstyle{definition}

\newtheorem{s}[theorem]{}
\newtheorem{remark}[theorem]{Remark}
\newtheorem{example}[theorem]{Example}

\theoremstyle{remark}

\begin{document}

\title[Resolutions of Koszul complexes]{Resolutions of Koszul complexes and applications}
\author{Tony~J.~Puthenpurakal}
\date{\today}
\address{Department of Mathematics, IIT Bombay, Powai, Mumbai 400 076}

\email{tputhen@math.iitb.ac.in}

 \begin{abstract}
In this paper we consider projective and injective resolutions of Koszul complexes and give several applications to the study of Koszul homology modules.
\end{abstract}
 \maketitle
\section{introduction}
The Koszul complex is a fundamental construction in commutative algebra.
W. Vasconcelos writes in his book ``Integral closures'' \cite[page 280]{VasBook-05}:
\begin{quote}
"While the vanishing of the homology of a Koszul complex $K(\bx, M)$ is easy to track, the module theoretic properties
of its homology, with the exception of the ends, is difficult to fathom. For instance, just trying to see
whether a prime is associated to some  $H_i(\bx, M)$  can be very hard."
\end{quote}
The purpose of this paper is to enhance our knowledge about Koszul homology by establishing the following results:

Let $(A,\m)$ be a commutative Noetherian local ring, let $I$ be an ideal in $A$ and let $M$ be a finitely generated $A$-module. Let $I$ be generated minimally by $\bu = u_1, \ldots, u_m$.
Denote by
$ \K(\bu, M)$  the Koszul complex associated to $\bu$ with coefficients in $M$. Set $H_i(\bu, M)$ the $i^{th}$ Koszul homology module of $M$ \wrt\ $\bu$.
If $M = R$ then we set $H_i(\bu) = H_i(\bu, R)$. As $\bu$ is a minimal generating set  of $I$ then  it is easily seen that $H_i(\bu, M)$ is an invariant of $I$ and $M$. In this case set $H_i(I, M) = H_i(\bu, M)$ and $H_i(I) = H_i(\bu)$. If $I = \m$ then set $H_i(A) = H_i(\m)$. It is also convenient to consider Koszul cohomology: let $\Kc(\bu, M)$ be the Koszul co-chain complex associated to $\bu$ with coefficients in $M$. We may consider Koszul cohomology modules $H^i(I, M)$.

Next we state the results proved in this paper.

 \textbf{I:} Let $(A,\m)$ be a Noetherian  local ring with residue field $k$. Then $H_1(A) \cong k $ if and only if $A$ is a hypersurface ring, i.e., the completion $\wh{A} = Q/(f)$ where $(Q, \n)$ is regular local and $f \in \n^2$; see \cite[2.3.2]{BH}. Recall a local ring $A$ is said to be analytically un-ramified if the completion $\wh{A}$ is reduced. Our first result is
\begin{theorem}\label{hyp}
Let $(A,\m)$ be an analytically un-ramified \CM \ local ring with embedding dimension $e$ and dimension $d$.  Assume $A$ is not regular. Then the following assertions are equivalent:
\begin{enumerate}[\rm (i)]
  \item $A$ is a hypersurface ring.
  \item $H_{e - d - 1}(A) \cong k$.
\end{enumerate}
\end{theorem}
We note that if $A$ is Gorenstein then the above results follow from the fact that $H_*(A)$ is a Poincare algebra, \cite[3.4.5]{BH}. To prove (ii) $\implies$ (i) we first show that $A$ is necessarily Gorenstein.

\textbf{II:} Recall $I = (\bu)$ is said to be an ideal of definition of $M$ if $\ell(M/IM)$ is finite. In this case $\ell(H_i(\bu, M)) $ is finite for all $i$. We may consider the Euler characteristic $\chi(\bu, M) = \sum_{i \geq 0}(-1)^i\ell(H_i(\bu, M))$. Serre proved that $\chi(\bu, M) \geq 0$ and is non-zero if and only if $\bu$ is a system of parameters of $M$, see \cite[4.7.6]{BH}. We prove
\begin{theorem}
\label{chi} Let $(A,\m)$ be a regular local ring and let $M$ be a finitely generated $A$-module. Let $\bu = u_1, \ldots, u_m$ and set $I = (\bu)$. Assume
$\ell(M\otimes A/I)$ is finite and $\dim M + \dim A/I < \dim A$. Then
\begin{enumerate}[\rm (1)]
  \item $\chi(\bu, M)  = 0$.
  \item $m > \dim M$.
\end{enumerate}
\end{theorem}

\textbf{III:} The total Koszul homology of $M$, i.e., $H_*(I, M)$ is a module over the algebra $H_*(I)$. So it is natural to expect that good properties of $H_*(I)$ will impose good properties of $H_*(I, M)$ under suitable conditions. Perfection of $H_i(I)$ has strong consequences. We prove

\begin{theorem}\label{perfect}
 Let $(A,\m)$ be a \CM  \ local ring and let $I$ be an ideal of $A$. Assume that $H_i(I)$ are perfect $A$-modules for $0 \leq i \leq r-1$ (whenever $H_i(I) \neq 0$).
 Let $M$ be a maximal \CM \ $A$-module. Then $\Ext^g(H_r(I), M) \cong H^{g+r}(I, M)$.  Furthermore if $A$ is Gorenstein and $H_i(I)$ is perfect for all $i$ (whenever $H_i(I) \neq 0$) then
 $H^i(I, M) \cong \Hom_A(M^*, H^i(I))$ for all $i$. In particular $H_i(I, M)$ is $S_2$ as an $A/I$-module.
\end{theorem}

\textbf{IV:}
Next we show that finiteness of projective (injective) dimensions of all but one Koszul homology modules implies the finiteness of projective (injective) dimension of the remaining one.
\begin{theorem}\label{all-one}
Let $(A,\m)$ be a Noetherian local ring. Let $M$ be a finitely generated $A$-module. Assume projective (injective) dimension of $M$ is finite. Let $I$ be an ideal in $A$.
Assume projective (injective) dimension of $H_j(I, M)$ is finite for all $j \neq i$. Then projective (injective) dimension of $H_i(I, M)$ is finite
\end{theorem}

In a different direction we show that homological properties of all $H_i(I, M)$ imposes condition on $M$. We prove
\begin{theorem}\label{growth}
Let $(A,\m)$ be a Noetherian local ring. Let $M$ be a finitely generated $A$-module.  Let $I$ be an ideal in $A$. Then
\begin{enumerate}[\rm (1)]
  \item  If $\projdim_A H_i(I, M)$ is finite for all $i$ then $\projdim_A M$ is finite.
  \item If $\injdim_A H_i(I, M)$ is finite for all $i$ then $\injdim_A M$ is finite.
\end{enumerate}
\end{theorem}

\textbf{V:} Now assume that $A$ is a local complete intersection. Then there is a notion of support variety $V(M)$ of an $A$-module $M$. We show:
\begin{theorem}
  \label{ci} Let $(A,\m)$ be a complete local complete intersection and let $I$ be an ideal of $A$. Assume $k = A/\m$ is algebraically closed. Let $M$ be a finitely generated $A$-module with $\projdim_A M < \infty$. Set $g =  \grade(I,M)$. Then for $i = 0, \ldots, \mu(I) - g$ we have
  $$ V(H_i(I, M)) \subseteq  \bigcup_{j \neq i} V(H_j(I, M)).$$
  In particular for every $i = 0,\ldots, \mu(I)-g$ we have
  $$\cx_A H_i(I, M) \leq \max \{ \cx_A H_j(I, M) \mid j \neq i \}.$$
\end{theorem}
We also show
\begin{theorem}
  \label{ci-m} Let $(A,\m)$ be a complete local complete intersection and let $I$ be an ideal of $A$. Assume $k = A/\m$ is algebraically closed. Let $M$ be a finitely generated \\ $A$-module.  We have
  $$ V(M) \subseteq  \bigcup_{ i} V(H_i(I, M)).$$
  In particular we have
  $$\cx_A M \leq \max \{ \cx_A H_j(I, M) \mid j  = 0, \ldots, \mu(I)-\grade(I) \}.$$
\end{theorem}

\textbf{VI:} \emph{Depth of Koszul homology modules:} Vanishing of Koszul homology modules gives information on depth of an $A$-module. However  depths of Koszul homology modules themselves are a bit mysterious. We first show
\begin{theorem}
\label{depth} Let $(A,\m)$ be a \CM \ local ring of dimension $d$ and let $I$ be an ideal of $I$. Set $m = \mu(I)$ and $g = \grade(I)$. Set $c = \min \{ \depth H_i(I) \colon 0 \leq i < m -g \}$. Then
\begin{enumerate}[\rm (1)]
  \item $\depth H_{m-g}(I)  \geq c$.
  \item If $c < d-g$ then $\depth H_{m-g}(I) \geq c + 1$.
\end{enumerate}
\end{theorem}
We also give an example (see \ref{ex-thurs}) which shows that the above result  may not hold if we do not assume $A$ is \CM.
When $A$ is regular we can improve Theorem \ref{depth}. Set $\cmd E = \dim E - \depth E$, the Cohen-Macaulay defect of $E$. We show
\begin{theorem}
\label{cmd} Let $(A,\m)$ be a regular local ring and let  $I$ be an ideal of $I$. Set $m = \mu(I)$ and $g = \grade(I)$. Set $c = \max \{ \cmd H_i(I) \colon 0 \leq i < m -g \}$. Then
$$ \cmd H_{m-g}(I) \leq \max \{ 0, c-2 \}. $$
\end{theorem}

The module  $H_{m-g}(I, M)$ behaves much better for a module $M$ of finite projective dimension. We show
\begin{theorem}
\label{pdim-depth} Let $(A,\m)$ be a \CM \ local ring of dimension $d$ and let $I$ be an ideal of $A$. Let $M$ be a finitely generated $A$-module. Set $m = \mu(I)$ and $g =\grade(I, M)$.  Then
\begin{enumerate}[\rm (1)]
  \item If $\projdim M \leq d-g-1$ then $\depth H_{m-g}(I, M) \geq 1$.
  \item If $\projdim M \leq d-g-2$ then $\depth H_{m-g}(I, M) \geq 2$.
\end{enumerate}
\end{theorem}

\textbf{VII:} Estimating depth of $\Hom_A(M, N)$ is usually a difficult job. As an application of our techniques we prove
\begin{theorem}\label{hom}
Let $(A,\m)$ be a \CM \ local ring.
Let $N$ be a maximal \CM \ $A$-module. Let $M$ be another finitely generated $A$-module with $\Hom_A(M, N) \neq 0$ and $\Ext^i_A(M, N) = 0$ for $i > 0$. Then
$\Hom_A(M, N)$ is a maximal \CM \ $A$-module.
\end{theorem}
Over local complete intersections it is not difficult to construct bountiful examples of modules satisfying  the hypotheses of the above Theorem, see \ref{const-mcm}.

\textbf{VIII:}
When $A$ is Gorenstein then we can prove some results without assuming $H_i(I)$ is perfect.
If $(A,\m)$ is Gorenstein and $I$ is a \CM \ ideal ,i.e., $R/I$ is \CM;  Vasconcelos notes
that $H_{l-g-1}(\by, R)$ is $S_2$; see \cite[1.3.2]{VasKH}. In the same paper \cite[2.2]{VasKH} he shows that there is a canonical map  $ H_{l-g-1}(\by) \rt \Ext_R^g(H_1(\by), R)$ which is an isomorphism  whenever $I$ is strongly \CM \  in codimension one. Recall an ideal $I$ is said to be strongly \CM \ if all the Koszul homology modules $H_i(\by)$ is \CM.

We show
\begin{theorem}
 \label{gor-cm} Let $(A,\m)$ be a Gorenstein local ring and let $I$ be an ideal in $A$ with grade $g$. Let $r \geq 1$ and assume $H_i(I)$ is \CM \ for $i = 0, \ldots, r-1$.
 Then $H^{r+g}(I) \cong \Ext^{g}_A(H_r(I), A)$. In particular $H^{r+g}(I)$ is $S_2$ as an $A/I$-module.
 \end{theorem}
 Perhaps the first case when we do-not know the depth of a Koszul homology module is when $\mu(I) - \grade(I) = 2$.
We  prove:
\begin{theorem}\label{two-add}
 Let $(A,\m)$ be a Gorenstein local ring of dimension $d$ and let $I$ be an ideal in $A$ with grade $g = \mu(I) - 2$. Assume $A/I$ is \CM. Then $\depth H_1(I) \geq d - g -2$.
\end{theorem}
 Gulliksen  \cite[1.4.9]{GL} proved that if $\projdim_A A/I$ is finite then $H_1(I)$ is a free $A/I$-module if and only if $I$ is generated by a regular sequence.
We note that  if $g = \grade (I)$ and $l = \mu(I)$ then $H_{l-g-1}(I)$ is  the dual of $H_1(I)$ in many cases. Let $A$ be Gorenstein local with $A/I$ \CM.  Then
$H_{l-g}(I)$ is the canonical module of $A/I$. In particular $\injdim_{A/I} H_{l-g}(I)$ is finite. We prove
\begin{corollary}\label{gul}
Let $(A,\m)$ be a Gorenstein local ring and let $I$ be an ideal in $A$ with $\projdim_A A/I$ finite. Set $g = \grade(I)$ and $l = \mu(I)$. Assume $l-g \geq 2$. Also assume $A/I$ and $H_1(I)$ is \CM.
Then $\injdim_{A/I} H_{l-g-1}(I) = \infty$.
\end{corollary}

\s \emph{Technique to prove our results:}
To prove result regarding modules, many a times we have to take projective or injective resolutions of the module. The main idea to prove our results is that we have to take projective or injective resolutions of the Koszul complex.

Let $\Kc(\bu, M)$ be the Koszul co-chain complex of $\bu = u_1, \ldots, u_m$ with coefficients in $M$. Let $\E$ be an injective resolution of $\Kc(\bu, M)$. Let $N$ be another finitely generated $A$-module.
We denote by $V^i(\bu , N, M)$ the $i^{th}$-cohomology of the co-chain complex $\Hom_A(N, \E)$. It can be shown that these cohomology modules are independent of the injective resolution chosen.
We use these cohomology modules to understand the cohomology modules $H^*(\bu, M)$. The modules $V^i(\bu, N, M)$ behave like Koszul homology modules. We prove
\begin{theorem}
  \label{prop}(with hypotheses as above)
  Set $I = (\bu)$. We have
  \begin{enumerate}[\rm (1)]
    \item $IV^*(\bu, N, M) = 0$. Furthemore $\ann M + \ann N \subseteq \ann V^*(\bu, N, M)$.
    \item  Set $\bu^\prime = u_1, \ldots, u_{m-1}$. Then we have a long exact sequence for all $i \in \Z$
    $$ \cdots \rt V^i(\bu,  N, M) \rt  V^i(\bu^\prime,  N, M)\xrightarrow{u_m} V^i(\bu^\prime, N, M) \rt V^{i+1}(\bu, N, M) \rt \cdots$$
  \end{enumerate}
\end{theorem}

Let $\K(\bu, M)$ be the Koszul chain complex of $\bu = u_1, \ldots, u_l$ with coefficients in $M$. Let $\Pb$ be a projective resolution of $\K(\bu, M)$. Let $N$ be another finitely generated $A$-module.
\begin{enumerate}
  \item We denote by $U^i(\bu , M, N)$ the $i^{th}$-cohomology of the co-chain complex $\Hom_A(\Pb, N)$. We can prove results for $U^*(\bu, M, N)$ which are analogus  to Theorem \ref{prop}.
  \item We denote by $W_i(\bu , M, N)$ the $i^{th}$-homology of the chain complex $\Pb\otimes N$. We can prove results for $W_*(\bu, M, N)$ which are analogus  to Theorem \ref{prop}.
\end{enumerate}

A spectral sequence relates Koszul (co)-homology with $V^*(\bu, N, M)$, \\  $U^*(\bu, M, N)$ and $W_*(\bu, M, N)$.
We show
\begin{theorem}
\label{spectral}(with notation as above).
\begin{enumerate}[\rm (1)]
  \item There exists a first quadrant cohomology spectral sequence
  $$ \Ext_A^p(N, H^q(\bu, M)) \Rightarrow V^{p+q}(\bu,  N, M).$$
  \item There exists a first quadrant  cohomology spectral sequence
  $$ \Ext_A^p(H_q(\bu, M), N) \Rightarrow U^{p+q}(\bu, M, N).$$
  \item There exists a first quadrant  homology spectral sequence
  $$ \Tor^A_p(H_q(\bu, M), N) \Rightarrow W_{p+q}(\bu, M, N).$$
\end{enumerate}
\end{theorem}
Here is an overview of the contents of this paper. In section two we describe our construction of projective (injective) resolutions of the Koszul complex and their homology groups. We also construct the three spectral sequences \ref{spectral}. In section three we give a proof of Theorem \ref{prop}, In section four we give a proof of Theorem \ref{hyp}. In the next section we give a proof of Theorem \ref{chi}. In section six we give a proof of Theorem \ref{perfect}. In the next section we give proofs of Theorems \ref{all-one} and \ref{growth}. In section eight we give proofs of Theorems \ref{ci} and \ref{ci-m}. In the next section we give proofs of Theorems \ref{depth}, \ref{cmd} and \ref{pdim-depth}. In section ten we give a proof of Theorem \ref{hom}. Finally in section eleven we give proofs of Theorems \ref{gor-cm}, \ref{two-add}.
\section{The construction and three spectral sequences}
In this section $A$ will denote a Noetherian ring. In this section we describe our constructions and three spectral sequences that we need.

\s  We let $D^b(A)$ denote the bounded derived category of $A$.
Let $K^{b, *}(\Proj A)$ be complexes $\Pb$ of finitely generated projective $A$ modules with $\Pb^n = 0$ for $n \gg 0$ and $H^*(\Pb)$ finitely generated.  Let
$K^{b, *}(\Inj A)$ be complexes $\E$ of injective $A$ modules with ${\E}^n = 0$ for $n \ll 0$ and $H^*(\E)$ finitely generated.  We note that the projective (injective) resolution functor yields an equivalence between $D^b(A)$ and $K^{b, *}(\Proj A)$ ($K^{b, *}(\Inj A)$).

\s \label{ind-V} Let $\Kc(\bu, M)$ be the Koszul co-chain complex of $\bu = u_1, \ldots, u_l$ with coefficients in $M$. Let $\E$ be an injective resolution of $\Kc(\bu, M)$. Let $N$ be another finitely generated $A$-module. We denote by $V^i(\bu , N, M)$ the $i^{th}$-cohomology of the co-chain complex $\Hom_A(N, \E)$. We have to show that these cohomology modules are independent of the injective resolution chosen. Let ${\E}^\prime$ be any other injective resolution of $\Kc(\bu, M)$. Then as $\E \cong {\E}^\prime$ in $D^b(A)$ we have maps $f \colon \E \rt {\E}^\prime$ and $g \colon {\E}^\prime  \rt \E$ in $D^b(A)$ with $g \circ f = 1_{\E}$ and $f\circ g = 1_{{\E}^\prime}$. We note that as $\E$ and ${\E}^\prime$ are complexes of injective $A$-modules it follows that $f, g$ are maps in $K^{b, *}(\Inj A)$. Therefore $\Hom_A(N, \E) \cong \Hom_A(N, {\E}^\prime)$. The result follows.

\s \label{f-ss} \emph{Construction of the first  spectral sequence:} Let $\Pb \rt N$ be a projective resolution of $N$ and let $\E$ be an injective resolution of $\Kc(\bu, M)$ with $E^n = 0$ for $n < 0$.
We consider the Hom co-chain double complex $\C = \Hom(\Pb, \E)$; see \cite[2.7.4]{Weibel}. Set
\[
\C = \{C^{pq} \}_{p,q\geq 0} \quad \text{where} \quad  C^{pq} = \Hom(P_p, E^q).
\]
Set $\T  = \Tot^{\bigoplus}(\C)$ where
$$ T^{n} = \bigoplus_{p+q = n}C^{pq}.$$

The cohomology of $\T$ can be computed by using two standard spectral sequences. We first consider the spectral sequence
whose ${}^{I}E_{1}^{pq}$ term is given by $ H^q(\Hom_A(P_p, \E))$.  Afterwards we consider the second standard spectral sequence whose ${}^{II}E_{1}^{pq}$  is given by $ H^{q}\left(\Hom_A(\Pb, E^p) \right)$.
We first prove
\begin{proposition}\label{first-ss}
We have
\begin{enumerate}[\rm (1)]
  \item ${}^{II}E_{2}^{pq}$-collapses and yields $H^n(\T) \cong V^n(\bu, N, M)$.
  \item ${}^{I}E_{2}^{pq} = \Ext^p_A(N, H^q(\bu, M))$.
\end{enumerate}
\end{proposition}
\begin{proof}
  (1) We have ${}^{II}E_{1}^{pq} = H^{q}\left(\Hom_A(\Pb, E^p)\right) = \Ext^q_A(N, E^p)$ as $E^p$ is an injective $A$-module. Thus
  ${}^{II}E_{1}^{pq} = 0$ for $q > 0$ and  ${}^{II}E_{1}^{p0} = \Hom_A(N, E^p)$ for $q = 0$. It follows that
   ${}^{II}E_{2}^{pq} = 0$ for $q > 0$ and  ${}^{II}E_{2}^{p0} = V^p(\bu, N, M)$ for $q = 0$. The result follows.

   (2) We have ${}^{I}E_{1}^{pq} = H^q(\Hom_R(P_p, \E))  = \Hom_A(P_p , H^q(\E)) $ (as $P_p$ is a projective $A$-module). Thus ${}^{I}E_{1}^{pq} = \Hom_A(P_p, H^q(\bu, M))$.
   It follows that ${}^{I}E_{2}^{pq} = \Ext^p_A(N, H^q(\bu, M))$.
\end{proof}

\s \label{ind-U} Let $\K(\bu, M)$ be the Koszul chain complex of $\bu = u_1, \ldots, u_l$ with coefficients in $M$. Let $\Pb$ be a projective resolution of $\K(\bu, M)$. Let $N$ be another finitely generated $A$-module. We denote by $U^i(\bu ,  M, N)$ the $i^{th}$-cohomology of the co-chain complex $\Hom_A(\Pb, N)$. By an argument similar to that in \ref{ind-V} we can show that these cohomology modules are independent of the projective resolution of $\K(\bu, M)$.

\s \label{s-ss} \emph{Construction of the second spectral sequence:} Let $\Pb \rt \K(\bu, M)$ be a projective resolution of $\K(\bu, M)$ with $P_n = 0$ for $n < 0$ and let $\E$ be an injective resolution of $N$.
We consider the Hom co-chain double complex $\D = \Hom(\Pb, \E)$; see \cite[2.7.4]{Weibel}. Set
\[
\D = \{D^{pq} \}_{p,q\geq 0} \quad \text{where} \quad  D^{pq} = \Hom(P_p, E^q).
\]
Set $\W  = \Tot^{\bigoplus}(\D)$ where
$$ W^{n} = \bigoplus_{p+q = n}D^{pq}.$$
The cohomology of $\W$ can be computed by using two standard spectral sequences. We first consider the spectral sequence
whose ${}^{I}E_{1}^{pq}$ term is given by \\  $ H^q(\Hom_A(P_p, \E))$.  Afterwards we consider the second standard spectral sequence whose ${}^{II}E_{1}^{pq}$  is given by $ H^{q}\left(\Hom_A(\Pb, E^p) \right)$.
We prove
\begin{proposition}\label{second-ss}
We have
\begin{enumerate}[\rm (1)]
  \item ${}^{I}E_{2}^{pq}$-collapses and yields $H^n(\W) \cong U^n(\bu, M, N)$.
  \item ${}^{II}E_{2}^{pq} = \Ext^p_A( H_q(\bu, M), N)$.
\end{enumerate}
\end{proposition}
\begin{proof}
(1) We have ${}^{I}E_{1}^{pq} =    H^q(\Hom_A(P_p, \E) = \Ext^q_A(P_p, N)$. It follows that ${}^{I}E_{1}^{pq} = 0$ for $q > 0$ and ${}^{I}E_{1}^{p0} = \Hom_A(P_p, N)$ for $q = 0$.
Thus ${}^{I}E_{2}^{pq} = 0 $ for $q > 0$ and ${}^{I}E_{2}^{pq} = U^p(\bu, M, N)$ for $q = 0$. The result follows.

(2) We have ${}^{II}E_{1}^{pq} =  H^{q}\left(\Hom_A(\Pb, E^p) \right) = \Hom_A( H_q(\Pb), E^p)$. Thus we have
${}^{II}E_{1}^{pq} = \Hom_A(H_q(\bu, M), E^p)$. It follows that ${}^{II}E_{2}^{pq} = \Ext^p_A( H_q(\bu, M), N)$.
\end{proof}
\s \label{ind-W} Let $\K(\bu, M)$ be the Koszul chain complex of $\bu = u_1, \ldots, u_l$ with coefficients in $M$. Let $\Pb$ be a projective resolution of $\K(\bu, M)$.  Let $N$ be another finitely generated $A$-module. We denote by $W_i(\bu ,  M, N)$ the $i^{th}$-homology of the chain complex $\Pb \otimes N$. By an argument similar to that in \ref{ind-V} we can show that these homology modules are independent of the projective resolution of $\K(\bu, M)$.

\s \label{t-ss} \emph{Construction of the third spectral sequence:} Let $\Pb \rt \K(\bu, M)$ be a projective resolution of $\K(\bu, M)$ with $P_n = 0$ for $n < 0$ and let $\F$ be a projective resolution of $N$.
We consider the double complex $\D = \Pb \otimes \F$. Set
\[
\D = \{D_{pq} \}_{p,q\geq 0} \quad \text{where} \quad  D_{pq} = P_p \otimes F_q.
\]
Set $\Y  = \Tot^{\bigoplus}(\D)$ where
$$ Y_{n} = \bigoplus_{p+q = n}D_{pq}.$$
The homology of $\Y$ can be computed by using two standard spectral sequences. We first consider the spectral sequence
whose ${}^{I}E^{1}_{pq}$ term is given by   $ H_q(P_p \otimes \F)$.  Afterwards we consider the second standard spectral sequence whose ${}^{II}E^{1}_{pq}$  is given by
$ H_{q}\left(\Pb \otimes F_p \right)$.

We prove
\begin{proposition}\label{third-ss}
We have
\begin{enumerate}[\rm (1)]
  \item ${}^{I}E^{2}_{pq}$-collapses and yields $H^n(\Y) \cong W_n(\bu, M, N)$.
  \item ${}^{II}E^{2}_{pq} = \Tor^A_p( H_q(\bu, M), N)$.
\end{enumerate}
\end{proposition}
\begin{proof}
(1) We have ${}^{I}E^{1}_{pq} = H_q(P_p \otimes \F) = \Tor^A_q(P_p, N) $. So we have ${}^{I}E^{1}_{pq} = 0$ for $q > 0$ and ${}^{I}E^{1}_{p0} = P_p\otimes N$.
Therefore ${}^{I}E^{2}_{pq} =  0$ for $q > 0$ and ${}^{I}E^{2}_{p0} = W_p(\bu, M, N)$. The result follows.

(2) We have ${}^{II}E^{1}_{pq} = H_q(\Pb \otimes F_p) = H_q(\bu, M)\otimes F_p$. So we have ${}^{II}E^{2}_{pq} = \Tor^A_p( H_q(\bu, M), N)$.
\end{proof}
\section{Proof of Theorem \ref{prop}}
In this section we give a proof of Theorem \ref{prop}. We need the following preliminary facts.
\s\label{cone} Let $\bu = u_1, \ldots,u_m$ and let $\bu^\prime = u_1, \ldots, u_{m-1}$. Then we note that $\K(\bu)$ is the mapping cone of the multiplication by $u_m$ on $\K(\bu^\prime)$. Thus we have an exact sequence of chain complexes
\[
0 \rt \K(\bu^\prime) \rt \K(\bu) \rt \K(\bu^\prime)[-1] \rt 0.
\]
This induces a map of chain complexes
\[
0 \rt \K(\bu^\prime, M) \rt \K(\bu, M) \rt \K(\bu^\prime, M)[-1] \rt 0,
\]
and a map of co-chain complexes
\[
0 \rt \Kc(\bu^\prime, M)[-1] \rt \Kc(\bu, M) \rt \Kc(\bu^\prime, M) \rt 0.
\]
We note that $\Kc(\bu, M)[1]$ is the mapping cone of the multiplication by $u_m$ on $\Kc(\bu^\prime, M)$.
\begin{proof}[Proof of Theorem \ref{prop}]
(1) Let $\E$ be an injective resolution of $\Kc(\bu, M)$ and let $\eta \colon \Kc(\bu, M) \rt \E$ be the corresponding map. Let $a \in I$ and let
$\mu^K_a \colon \Kc(\bu, M) \rt \Kc(\bu, M)$ be multiplication by $a$. Let $\mu^E_a \colon \E \rt \E$ be multiplication by $a$. We have $\eta \circ \mu^K = \mu^E \circ \eta$.
By \cite[1.6.5]{BH} we have $\mu^K$ is null-homotopic. So $\eta \circ \mu^K = 0$ in $D^b(A)$. It follows that $\mu^E \circ \eta = 0$ in $D^b(A)$. But $\eta$ is an invertible map in $D^b(A)$.
So $\mu^E = 0$ in $D^b(A)$. However $\E$ is a complex of injectives with $E^n = 0$ for $n \ll 0$. It follows that $\mu^E$ is null-homotopic. Thus multiplication by $a$ in the complex $\Hom_A(N, \E)$ is also null-homotopic. The result follows.

The fact that $\ann M + \ann N \subseteq \ann V^*(\bu, N, M)$ is elementary and  is left to the reader.

(2) Let ${\E}^\prime $ be an injective resolution of $\Kc(\bu^\prime, M)$. The (shift of) Koszul co-chain complex $\Kc(\bu, M)[1]$ is the mapping cone of multiplication by $u_m$ on  $\Kc(\bu^\prime, M)$. It is readily verified that the mapping cone  of multiplication by $u_m$ on ${\E}^\prime $ is an injective resolution $\E$ of $\Kc(\bu, M)[1]$. We have an exact sequence
of complexes $0 \rt {\E}^\prime \rt \E \rt {\E}^\prime[1] \rt 0$. Applying $\Hom_A(N, -)$ and taking the long exact sequence in cohomology yields the required result.
\end{proof}

The following  two results can be proved in a similar way as in Theorem \ref{prop}
\begin{theorem}
  \label{prop-U}(with hypotheses as above)
  Set $I = (\bu)$. We have
  \begin{enumerate}[\rm (1)]
    \item $IU^*(\bu, M, N) = 0$. Furthemore $\ann M + \ann N \subseteq \ann U^*(\bu, M, N)$.
    \item  Set $\bu^\prime = u_1, \ldots, u_{m-1}$. Then we have a long exact sequence for all $i \in \Z$
    $$ \cdots \rt U^i(\bu,  M, N) \rt  U^i(\bu^\prime,  M, N)\xrightarrow{u_m} U^i(\bu^\prime, M, N) \rt U^{i+1}(\bu, M, N) \rt \cdots$$
  \end{enumerate}
\end{theorem}
\begin{theorem}
  \label{prop-W}(with hypotheses as above)
  Set $I = (\bu)$. We have
  \begin{enumerate}[\rm (1)]
    \item $IW_*(\bu, M, N) = 0$. Furthemore $\ann M + \ann N \subseteq \ann W_*(\bu, M, N)$.
    \item  Set $\bu^\prime = u_1, \ldots, u_{m-1}$. Then we have a long exact sequence for all $i \in \Z$
    $$ \rt W_i(\bu,  M, N) \rt  W_{i-1}(\bu^\prime,  M, N)\xrightarrow{u_m} W_{i-1}(\bu^\prime, M, N) \rt W_{i-1}(\bu, M, N) \rt$$
  \end{enumerate}
\end{theorem}

\begin{remark}
  It is convenient to set $\K(\emptyset, M) = \Kc(\emptyset, M) = M$. If $\bu = u$ then we can set $\bu^\prime = \emptyset$. Then note that the long exact sequence in \ref{prop}, \ref{prop-U}
  and \ref{prop-W} still holds.
\end{remark}

\section{Proof of Theorem \ref{hyp}}
In this section we give a proof of Theorem \ref{hyp}. We need to prove a few preliminary results. Let $\mu_i(\m, M) = \dim \Ext^i_A(k. M)$ denote the $i^{th}$ Bass number of $M$ with respect to $\m$.
We first prove
\begin{proposition}\label{bass}
Let $(A, \m)$ be a \CM \ local ring and let $M$ be an MCM $A$-module. Set $k = A/\m$. Assume $\injdim M = \infty$. Let $\bu = u_1, \ldots, u_m$. Then
$$ \dim_k V^i(\bu, k, M) = \begin{cases}
                             0, & \mbox{if } i < d \\
                             \mu_{d}(\m, M), & \mbox{if }  i = d \\
                             m\mu_{d}(\m, M) + \mu_{d+1}(\m, M), & \mbox{if }  i = d+1.
                           \end{cases}$$
\end{proposition}
\begin{proof}
We prove by induction on $m$. When $m = 0$ then note that $V^i(\emptyset, k, M) = \Ext^i_A(k, M)$. In this case the result clearly holds.

We assume the result for $\bu^\prime = u_1, \ldots, u_{m-1}$ and prove for $\bu$.
By \ref{prop} we have an exact sequence for all $i$
\[
0 \rt V^{i-1}(\bu^\prime, k, M) \rt V^{i}(\bu, k, M) \rt V^{i}(\bu^\prime, k, M) \rt 0.
\]
The result follows by induction.
\end{proof}
Next we show
\begin{proposition}\label{bass-mu}
Let $(A, \m)$ be a \CM \ local ring and let $M$ be an MCM $A$-module. Set $k = A/\m$, $d = \dim A$ and $e = \embdim(A)$. Assume $\injdim M = \infty$. Let $\bu = u_1, \ldots, u_e$ generate $\m$ minimally. Then
$$ \mu_{d+1}(\m, M) \leq \dim_k H^{d+1}(\m, M). $$
\end{proposition}
\begin{proof}
Let $\m = (u_1, \ldots, u_e)$. By \ref{first-ss} we have
\[
 {}^{I}E_{2}^{pq} = \Ext^p_A(k, H^q(\bu, M))  \Rightarrow V^{p+q}(\bu, k, M).
\]
We note that
\[
{}^{I}E_{2}^{1, d} = {}^{I}E_{\infty}^{1, d}.
\]
So $\dim {}^{I}E_{\infty}^{1, d} = e\mu_d(\m, M)$.
We also  have an exact sequence
\[
0 \rt {}^{I}E_{3}^{0, d+1} \rt  {}^{I}E_{2}^{0, d+1} \rt   {}^{I}E_{2}^{2, d}.
\]
It is readily verified that
\[
{}^{I}E_{3}^{0, d+1} = {}^{I}E_{\infty}^{0, d+1}.
\]
Thus $\dim {}^{I}E_{\infty}^{0, d + 1} \leq \dim_k H^{d+1}(\m, M)$.
Notice when $p + q = d +1$ and $p \neq 0, 1$ then  ${}^{I}E_{2}^{pq} = 0$.
So
we have an exact sequence
\[
0 \rt {}^{I}E_{\infty}^{1, d}  \rt V^{d+1}(\m, k, M) \rt {}^{I}E_{\infty}^{0, d+1} \rt 0.
\]
Thus we have by \ref{bass}
\[
e\mu_d(\m, M) + \mu_{d+1}(\m, M) \leq e\mu_d(\m, M) + \dim_k H^{d+1}(\m, M).
\]
The result follows.
\end{proof}
Finally we give
\begin{proof}[Proof of Theorem \ref{hyp}]
We first show that $A$ is  Gorenstein.  We may assume that $A$ is  complete. Suppose if possible $A$ is \emph{not} Gorenstein. Let $\omega$ be the canonical module of $A$. As $A$ is reduced we have that $\omega$ has a rank, see \cite[3.3.18]{BH}. Note $\rank \omega = 1$.

As $A$ is not Gorenstein we have $\injdim A = \infty$.
By \ref{bass-mu} we have $\mu_{d+1}(\m, A) \leq 1$. By an exercise problem, \cite[3.5.12(b)]{BH}, we get  $\mu_{j}(\m, A) \neq 0$ for all $j > d$. Thus  $\mu_{d+1}(\m, A)  = 1$.

Let $E$ be the injective hull of $k$.
Let $\Gamma_\m(-)$ be the $\m$-torsion functor. Then we have $H^i_m(A) = 0$ for $i <  d$ and $i > d$. We also have $H^d_\m(A) \neq 0$.
Let $\I$ be the minimal injective resolution of $A$. Set $a_i = \mu_i(\m, M)$. Taking $\Gamma_\m(\I)$ we get an exact sequence
$$ 0 \rt H^{d}_\m(A)  \rt E^{a_d} \rt E \rt E^{a_{d+2}} \rt \cdots$$
Dualizing  with respect to $E$ we obtain an exact sequence
\[
\cdots A^{a_{d+2}} \xrightarrow{\phi} A \rt A^{a_d} \rt \omega \rt 0.
\]
We note $\projdim \omega = \infty$.
It is elementary to note that $\rank \image \phi = 1$. So $\rank \coker \phi  = 0$ and it is non-zero. It follows that $A^{a_d}$ has a submodule of dimension $< d$ which is a contradiction.
Thus $A$ is Gorenstein.

We now have to show $A$ is a hypersurface. Suppose it is not. Then note $e = \embdim(A) \geq d + 2$. We also have that $H_*(A)$ is a Poincare algebra, see \cite[3.4.5]{BH}. It follows that $H_1(A) \cong k$. It follows that $A$ is a hypersurface ring.
\end{proof}

\section{Proof of Theorem \ref{chi}}
In this section we assume $A$ is a regular local ring. Recall Serre's vanishing Theorem, see \cite[p.\ 106]{S} for the equi-characteristic case and
\cite[5.6]{GS} and \cite[Theorem 1]{R} in general;
Let $M, N$ be finitely gnerated $A$-modules with $M \otimes N$ having finite length.
If $\dim  M + \dim N < \dim A$ then $$\sum_{i \geq 0}(-1)^i\ell(\Tor^A_i(M, N)) = 0.$$
We now give
\begin{proof}[Proof of Theorem \ref{chi}]
We note that (i) and (ii) are equivalent by Serre's Theorem \cite[4.7.6]{BH}. We prove (i).
Let $I = (u_1, \ldots, u_m)$. Let $\K = \K(\bu, A)$. We may take the projective resolution of $\K$ to be $\K$ itself. It follows that
$W_i(\bu, A, M) = H_i(\K(\bu, M))$ for $i \geq 0$. By \ref{third-ss} we have a homology spectral sequence
$$E^2_{pq} = \Tor^A_p(H_q(I), M) \Rightarrow H_{p+q}(\bu, M). $$
We note that $H_q(I)$ are $A/I$-modules and so $\dim M + \dim H_q(I) < \dim A$. So by Serre vanishing Theorem we have
$\sum_{p \geq 0}(-1)^p\ell(\Tor^A_p(H_q(I), M)) = 0.$
So we have
\[
\sum_{p,q} (-1)^{p+q} \ell(E^2_{pq}) = \sum_{q} (-1)^{q} \left(\sum_p (-1)^p\ell(E^2_{pq}) \right) = 0.
\]
By preservation of Euler characteristics in a spectral sequence, \cite[1.F, p.\ 15]{M},  we have
\[
\sum_{p,q} (-1)^{p+q} \ell(E^2_{pq}) = \sum_{p,q} (-1)^{p+q} \ell(E^\infty_{pq})
\]
We have
$$\sum_{i \geq 0}(-1)^i \ell(H_i(\bu, M)) = \sum_{i \geq 0} \sum_{p+q = i}(-1)^{p+q} \ell(E^\infty_{pq})  = 0.$$
The result follows.
\end{proof}
\section{Proof of Theorem \ref{perfect}}
In this section we give a proof of Theorem \ref{perfect}. The proof splits into two parts. In the first part we do not assume $A$ is Gorenstein. In the second part we assume that $A$ is Gorenstein.

\begin{proof}[Proof of Theorem \ref{perfect}, first part]
Let $I = (u_1, \ldots, u_m)$ (minimally). Set $\K = \K(\bu, A)$. We take the projective resolution of $\K$ to be itself.
It follows that $U^i(\bu, A, M) = H^i(\bu, M)$.   By \ref{second-ss} we have a convergent spectral sequence
\[
E_2^{p,q} = \Ext^p_A(H_q(\bu), M) \Rightarrow H^{p+q}(\bu, M).
\]
Note that $H_q(\bu, A)$ are annhilated by $I$. So $\Ext^i_A(H_q(\bu), M) = 0$ for $i < g$. We also have that as $H_i(\bu)$ is perfect for $i \leq r-1$ it follows that $\Ext^j(H_i(\bu), M) = 0$ for $j > g$ (whenever $i \leq r -1$).
So we have
\begin{enumerate}
  \item $E_2^{g,r} = \Ext_A^g(H_r(\bu), M) = E_\infty^{g,r}$.
  \item  Let $p + q = g +r$.
  \begin{enumerate}
    \item If $q \leq r-1$ then $p \geq g +1$. So $E_2^{p,q} = 0$.
    \item   If $q \geq r + 1$ then $p \leq g -1$. So $E_2^{p,q} = 0$.
  \end{enumerate}
\end{enumerate}
It follows that $H^{r+g}(I, M) \cong E^{g,r}_\infty = \Ext_A^g(H_r(\bu), M)$.
\end{proof}

Next we prove the second part of Theorem \ref{perfect}. We need a few preliminary results.
\begin{lemma}\label{p-gor}
Let $(A,\m)$ be a Gorenstein local ring and let $M$ be a MCM $A$-module. Let $I = (\bu)$ be an ideal in $A$ with $\projdim_A H^j(I)$ finite for all $j$. Then $V^j(\bu, M, A) = \Hom_A(M, H^j(I))$ for all $j \geq 0$.
\end{lemma}
\begin{proof}
We have by \ref{first-ss} that $E_2^{pq} = \Ext_A^p(M, H^q(\bu)) \Rightarrow V^{p+q}(\bu, M, A)$.
As \\ $\projdim H^q(\bu)$ is finite we get that $E_2^{pq} = 0 $ for $p > 0$. So the spectral sequence collapses. The result follows.
\end{proof}
The second preliminary result that we need is the following:
\begin{lemma}\label{vext}
Let $A$ be a Noetherian  ring and let $M, N$ be  $A$-modules. Assume $\Ext^i_A(M, N) = 0$ for $i > 0$. Then we have an isomorphism \\
$V^i(\bu, M, N)  \cong H^i(\bu, \Hom_A(M, N))$.
\end{lemma}
To prove this we need the following easily proved fact.
\begin{proposition}\label{d-ex}
Let $\C$ be a bounded below co-chain complex (i.e., $\C^n = 0$ for $n \ll 0$). Assume $H^*(\C)  = 0$. Suppose $M$ is an $A$-module with $\Ext^i_A(M, \C^n) = 0$ for all $i > 0$
and for all $n \in \Z$. Then $H^*(\Hom_A(M, \C)) = 0$.
\end{proposition}
We now give
\begin{proof}[Proof of Lemma \ref{vext}]
Let $\Kc = \Kc(\bu, N)$. Let $\E$ be an injective resolution of $\Kc$ and let $\eta \colon \Kc \rt \E$ be a quism and set $\C = \cone(\eta)$.
We have a short exact sequence of complexes $0 \rt \E \rt \C \rt \Kc[1] \rt 0$. As $\E$ is a complex of injective's we have an exact sequence of complexes
$$0 \rt \Hom_A(M, \E) \rt \Hom_A(M, \C) \rt \Hom_A(M, \Kc)[1] \rt 0.$$
We note that $\C$ is acyclic bounded below co-chain complex and that $\C^n$ is a direct sum of an injective $A$-module and some copies of $N$. So $\Ext^i_A(M, \C^n) = 0$ for all $i > 0$ and for all $n \in \Z$. So by \ref{d-ex}
the complex $\Hom_A(M, \C)$ is also acyclic. It follows that
\[
V^i(\bu, M, N) = H^i(\Hom(M, \E)) \cong H^i(\Hom(M, \Kc)) \quad \text{for all} \ i \geq 0.
\]
We note that
\begin{align*}
   \Hom_A(M, \Kc)&= \Hom(M, \Hom(\K(\bu), N)) \\
   &\cong \Hom(M \otimes \K(\bu), N)   \\
   &\cong \Hom(\K(\bu), \Hom(M, N)).
\end{align*}
The result follows.
\end{proof}
We now give
\begin{proof}[Proof of Theorem \ref{perfect}, second part]
By \ref{p-gor} we get $$\Hom_A(M, H^j(\bu)) \cong V^j(\bu, M, A).$$  As $M$ is MCM $A$-module and as $A$ is Gorenstein we have $\Ext^i_A(M, A) = 0$ for $i > 0$. So by \ref{vext} we get
 $V^j(\bu, M, A) \cong H^j(\bu, M^*)$. It follows that \\  $\Hom_A(M, H^j(\bu) ) \cong H^j(\bu, M^*)$.
 Replacing $M$ by $M^*$ yields the result.
\end{proof}

\section{Proof of Theorems \ref{all-one} and \ref{growth}}
We first give
\begin{proof}[Proof of Theorem \ref{all-one}]
We first consider the case when $\projdim_A M$ is finite and $\projdim_A H_j(\bu, M) < \infty$ for $j \neq i$.
As $\projdim_A M$ is finite, we get $W_i(\emptyset, M, k) = 0$ for $i \gg 0$. It follows from \ref{prop-W} that $W_i(\bu, M, k) = 0$ for $i \gg 0$ (say $i \geq s$).
Suppose $H_j(\bu, M)$ has finite projective dimension (over $A$) for all $j \neq i$. Recall we have a convergent spectral sequence
$$ E^{p,q}_2 = \Tor^A_p(H_q(I, M), k) \Rightarrow W_{p+q}(\bu, M, k).$$
For $p \geq d +1$ and $q \neq i$, note $E^2_{p,q} = 0$.

Assume $i + a = m = \mu(I)$. Let $p \geq d + 3 + m + s$. We first note that we have a sequence
\[
0 = E^2_{p+2, i -1} \rt E^2_{p,i} \rt E^2_{p-2, i+1} = 0.
\]
So we have $E^3_{p,i} = E^2_{p,i}$. Similarly we have
\[
0 = E^3_{p+3, i -2} \rt E^3_{p,i} \rt E^3_{p-3, i+2} = 0.
\]
So we have $E^4_{p,i} = E^3_{p,i} = E^2_{p,i}$. Iterating this procedure and noting that $E^r_{*, t} = 0$ for $t \geq m$ it follows that
$E^\infty_{p,i} = E^2_{p,i}$. As $W_{p+i}(\bu, M, k) = 0$ it follows that $E^2_{p,i} = 0$. So $\projdim_A H_i(\bu, M)$ is finite.

Next we consider the case when $\injdim_A M$ is finite and $\injdim_A H^j(\bu, M) < \infty$ for $j \neq i$.
As $\injdim_A M$ is finite we get $V^i(\emptyset, k, M) = 0$ for $i \gg 0$. It follows from \ref{prop} that $V^i(\bu, k, M) = 0$ for $i \gg 0$, say $i \geq s$.
We have a convergent spectral sequence
$$ E^{p,q}_2 = \Ext_A^p(k,  H^q(\bu, M)) \Rightarrow V^{p+q}(\bu, k, M).$$
For $p \geq d + 1$ and $q \neq i$ we have $E^{p,q}_2 = 0$.

Assume $i + a = m = \mu(I)$. Let $p \geq d + 3 + m + s$. Then by an argument similar to above we can show $E^{p,i}_\infty = E^2_{p,i}$. As $V^{p+q}(\bu, k, M) = 0$ it follows that $E^2_{p.i} = 0$. So $\injdim_A H^i(\bu, M) $ is finite.
\end{proof}
Next we give
\begin{proof}[Proof of Theorem \ref{growth}]
First assume $\projdim H_i(\bu, M)$ is finite for all $i$ where $\bu = u_1, \ldots, u_m$.
We have a convergent spectral sequence
$$ E^{p,q}_2 = \Tor^A_p(H_q(\bu, M), k) \Rightarrow W_{p+q}(\bu, M, k).$$
Assume $p +q \geq d + m + 2$. If $ p \geq d +1$ we get $E^{p,q}_2 = 0$ (as $\projdim_A H_i(\bu, M)$ is finite for all $i$). If $p \leq d$ then $q \geq m + 2$. So $H_q(I, M) = 0$. In particular $E_2^{p, q} = 0$. It follows that $W_n(\bu, M, k) = 0$ for $n \geq d + m + 2$. It follows from  \ref{prop-W} that $W_n(\emptyset, M, k) = 0$ for $n \geq d + m + 2$.
Thus $\projdim_A M$ is finite.

Next assume $\injdim_A H^j(\bu, M)$ is finite for all $j$. We have a convergent spectral sequence
$$ E^{p,q}_2 = \Ext_A^p(k,  H^q(\bu, M)) \Rightarrow V^{p+q}(\bu, k, M).$$
Assume $p + q \geq d + m +2$. If $p \geq d +1$ then $E^{p,q}_2 = 0$ as $\injdim_A H^j(\bu, M) $ is finite. If $p \leq d$ then $q \geq m + 2$. So $E^2_{p,q} = 0$. It follows that
$V^n(\bu, k, M) = 0$ for $n \geq d + m +2$. It follows from \ref{prop} that $V^n(\emptyset, k, M) = 0$ for $n \geq d + m +2$. Thus $\injdim_A M$ is finite.
\end{proof}

\section{Proof of Theorems \ref{ci} and \ref{ci-m}}
In this section $A$ is a local, complete, complete intersection with algebraically closed residue field $k = A/\m$. (Our results can be proved more generally , but we consider $k$ to be algebraically closed for simplicity). Let $(Q,\n)$ be a complete regular local ring such that $A = Q/(\fb)$ where $\fb = f_1, \ldots, f_r \in \n^2$ is a $Q$-regular sequence.

\s Let $M, N$ be  $A$-modules and let $\mathbb{F}$ be a resolution of $M$. Then we can construct the Eisenbud operators $t_1, \ldots, t_r$ on $\mathbb{F}$, see \cite{Eisenbud}. This turns $\Ext_A^*(M, N) = \bigoplus_{n \geq 0}\Ext^i_A(M, N)$ into a $R = A[t_1, \ldots, t_r]$-module with $\deg t_i = 2$ for all $i$. Gulliksen  \cite[3.1]{Gulliksen} showed that $\Ext_A^*(M, N) $ is a finitely generated $R = A[t_1, \ldots, t_r]$-module.
In particular \\ $\Ext^*_A(M, k)$ is a finitely generated graded $S = k[t_1, \ldots, t_r]$-module. Let $I(M)$ be the annhilator of this module. We set
$V(M)$ to be variety in $\mathbb{P}^{r-1}(k)$ defined by $I(M)$.

We first give
\begin{proof}[Proof of Theorem \ref{ci}]
Let $a \in \cap_{j \neq i}V(H_j(I, M))^c$. Choose an MCM $A$-module $N$ with $V(N) = \{ a \}$, see \cite[2.3]{B}.

As $\projdim_A M$ is finite, we get $W_i(\emptyset, M, N) = 0$ for $i \gg 0$. It follows from \ref{prop-W} that $W_i(\bu, M, N) = 0$ for $i \gg 0$ (say $i \geq s$).

Recall we have a convergent spectral sequence
$$ E^{p,q}_2 = \Tor^A_p(H_q(I, M), N) \Rightarrow W_{p+q}(\bu, M, N).$$
As $a \notin V(H_j(I,M))$ for $j \neq i$ it follows from \cite[Theorems III, IV]{avr-b} that \\  $\Tor^A_r(H_j(I, M), N) = 0$ for $r > d$. For $p \geq d +1$ and $q \neq i$, note $E^2_{p,q} = 0$.

Assume $i + a = m = \mu(I)$. Let $p \geq d + 3 + m + s$. We first note that we have a sequence
\[
0 = E^2_{p+2, i -1} \rt E^2_{p,i} \rt E^2_{p-2, i+1} = 0.
\]
So we have $E^3_{p,i} = E^2_{p,i}$. Similarly we have
\[
0 = E^3_{p+3, i -2} \rt E^3_{p,i} \rt E^3_{p-3, i+2} = 0.
\]
So we have $E^4_{p,i} = E^3_{p,i} = E^2_{p,i}$. Iterating this procedure and noting that $E^r_{*, t} = 0$ for $t \geq m$ it follows that
$E^\infty_{p,i} = E^2_{p,i}$. As $W_{p+i}(\bu, M, N) = 0$ it follows that $E^2_{p,i} = 0$.  Thus for every $p \geq d + 3 + m + s$ we  have $ \Tor^A_p(H_i(I, M), N) = 0$.
So $a \notin V(H_i(I, M))$. The result follows.
\end{proof}
Next we give
\begin{proof}[Proof of Theorem \ref{ci-m}]
Let $a \in \cap_{j }V(H_j(I, M))^c$. Choose an MCM $A$-module $N$ with $V(N) = \{ a \}$, see \cite[2.3]{B}.

Recall we have a convergent spectral sequence
$$ E^{p,q}_2 = \Tor^A_p(H_q(I, M), N) \Rightarrow W_{p+q}(\bu, M, N).$$
As $a \notin V(H_j(I,M))$ for all $j $ it follows from \cite[Theorems III, IV]{avr-b}  that  \\ $\Tor^A_r(H_q(I, M), N) = 0$ for $r > d$.

Let $p + q \geq d + 1 + m$. When $ p \geq d +1$ we get $E^{p,q}_2 = 0$ as $\Tor^A_r(H_q(I, M), N) = 0$ for $r > d$. When $p \leq d $ then $q \geq m +1$. So $H_q(I, M) = 0$. In particular $E_2^{p,q} = 0$. It follows that for every $n = p+q \geq d + 1 + m$ we have $W_{n}(\bu, M, N) = 0$. Let $\bu^\prime = u_1, \ldots, u_{m-1}$ .  By \ref{prop-W}  it follows that for  $n  \geq d + 1 + m$ we have a surjective  map
$$W_{n}(\bu^\prime, M, N) \xrightarrow{u_m} W_{n}(\bu^\prime, M, N).$$
By Nakayama's lemma it follows that $W_{n}(\bu^\prime, M, N) = 0$. Iterating we get \\
$W_{n}(\emptyset, M, N) = 0$ for all   $n  \geq d + 1 + m$. It follows that $\Tor^A_n(M , N) = 0$ for $n \gg 0$. Therefore $a \notin V(M)$. The result follows.
\end{proof}

\section{Proof of Theorems \ref{depth}, \ref{cmd} and  \ref{pdim-depth}  }
In this section we give proofs of Theorems \ref{depth}, \ref{cmd} and \ref{pdim-depth}.
We first give
\begin{proof}[Proof of Theorem \ref{depth}]
As $A$ is \CM \ it follows that \\ $V^i(\emptyset, k, A) = \Ext_A^i(k, A) = 0$ for $i < d$. It follows from \ref{prop}  that $V^i(\bu, k, A) = 0$ for $i <  d$. Set $c  =  \min \{ \depth H^i(I) \mid g < i \leq m$. We have a convergent spectral sequence
\[
E^{p,q}_2 = \Ext^p_A(k, H^q(\bu)) \Rightarrow V^{p+q}(\bu, k, A).
\]
Suppose if possible $r = \depth H^g(\bu) < c$.
Then notice that $E^{r, g}_\infty = E^{r,g}_2 = \Ext^r_A(k, H^g(\bu)) \neq 0$. However
$r + g < c + g \leq d -g + g = d$. But this contradicts the fact that $V^i(\bu, k, A) = 0$ for $i < d$. The result follows.

Next we consider the case when $c < d -g$. By earlier result we get that $\depth H^g(I) \geq c$. Suppose if possible $\depth H^g(I) = c$. Then notice
$E^{c, g}_\infty = E^{c,g}_2 = \Ext^c_A(k, H^g(\bu)) \neq 0$. It follows that $V^{c+g}(\bu, k, A) \neq 0$. This is a contradiction as $c+g < d$.
\end{proof}
Next we give an example which shows that Theorem \ref{depth} is false if we do not assume $A$ to be  \CM.
\begin{example}\label{ex-thurs}
Let $A = k[[X, Y]]/(XY, Y^2) = k[[x, y]]$ (here $x, y$ denotes the images of $X, Y$ in $A$).
Consider $K(y, A) \colon 0 \rt A \xrightarrow{y} A \rt 0$. We have $H_0(y, A) = k[[X]]$ which has depth $1$. We note that $y \in H_1(y, A)$ and $\m$ annihilates $y$. So $\depth H_1(y, A) = 0$.
\end{example}

Next we give
\begin{proof}[Proof of Theorem \ref{cmd}]
Set $\cmd H_i(I) = c_i$. By Auslander-Buchsbaum formula it follows that $\projdim_A H_i(I) = g  + c_i$. Thus $\projdim H_i(I) \leq g + c$ for all $i < m - g$.

We take the projective resolution of $\K = \K(I)$ to be $\K$-itself. So $U^i(I, A,k)
= H^i(I, k)$. We have a convergent spectral sequence
\[
E_2^{p,q} = \Ext^p_A(H_q(I), k) \Rightarrow U^{p+q}(I, A,k).
\]
We note that $E^{p,q}_2 = 0$ if
\begin{enumerate}
\item
$q > m -g$.
\item
$p > g +c$ and $q \neq m -g$.
\end{enumerate}
 It follows that
 $$\Ext^{g + c -1}_A(H_{m-g}(I), k) = E^{g + c -1, m-g}_2 = E^{g + c -1, m-g}_\infty.$$

  If $c \geq 2$ then $g + c - 1 + m -g \geq m +1$. So $E^{g + c -1, m-g}_\infty$ which is a sub-quotient of $U^{n}(I, A,k)$ with $n \geq m +1$. The latter module is zero. So
  $\Ext^{g + c -1}_A(H_{m-g}(I), k) = 0$. Thus $\projdim_A H_{m-g}(I) \leq g + c -2$. By
  Auslander-Buchsbaum formula it follows that $\cmd H_{m-g}(I) \leq c -2$.

   If $c \leq 1$ then notice
   \[
   \Ext^{g + 1}_A(H_{m-g}(I), k) = E^{g + 1, m-g}_2 = E^{g + 1, m-g}_\infty.
   \]
   As $U^{m + 1}(I, A, k) = 0$ it follows that  $\Ext^{g + 1}_A(H_{m-g}(I), k) = 0$.
   It follows that $\projdim_A H_{m-g}(I) \leq g$. By
  Auslander-Buchsbaum formula it follows that \\ $\cmd H_{m-g}(I)  = 0$. The result follows.
\end{proof}

Finally we prove
\begin{proof}[Proof of Theorem \ref{pdim-depth}]
Let $\bx$ be a parameter ideal in $A$.

(1) We have $W_i(\emptyset, M, A/(\bx)) = \Tor^A_i(M, A/(\bx)) = 0$ for $i \geq d-g$. It follows that $W_i(u_1, M, A/(\bx)) = 0$ for $i \geq d- g +1$.
Iterating we obtain $W_i(\bu, M, A/(\bx)) = 0$ for $i \geq d - g + m$.

We have a convergent spectral sequence
\[
E^2_{p,q} = \Tor^A_p(H_q(\bu, M), A/(\bx)) \Rightarrow W_{p+q}(\bu, M, A/(\bx))
\]
As $\projdim_A A/(\bx) = d$ we obtain $E^2_{p,q} =0$ for all $p > d$. We also obtain $E^2_{p,q} = 0$ for $q > m - g$.
So we obtain
$$ E^\infty_{d, m -g} = E^2_{d, m -g} = \Tor^A_d(H_{m-g}(I, M), A/(\bx)). $$
But $W_{d+m-g}(\bu, M, A/(\bx)) = 0$. So $\Tor^A_d(H_{m-g}(I, M), A/(\bx)) = 0$. As $\bx$ is a regular sequence we obtain $H_d( \bx, H_{m-g}(I, M)) = 0$. It follows  that\\ $\depth H_{m-g}(I, M) > 0$.
\end{proof}

(2) An argument similar to (1) shows that  $W_i(\bu, M, A/(\bx)) = 0$ for $i \geq d - g + m -1$.
We have a convergent spectral sequence
\[
E^2_{p,q} = \Tor^A_p(H_q(\bu, M), A/(\bx)) \Rightarrow W_{p+q}(\bu, M, A/(\bx))
\]
As $\projdim_A A/(\bx) = d$ we obtain $E^2_{p,q} =0$ for all $p > d$. We also obtain $E^2_{p,q} = 0$ for $g > m - g$.
So we obtain
$$ E^\infty_{d-1, m -g} = E^2_{d-1, m -g} = \Tor^A_{d-1}(H_{m-g}(I, M), A/(\bx)). $$
But $W_{d+m-g -1}(\bu, M, A/(\bx)) = 0$. So $\Tor^A_{d -1}(H_{m-g}(I, M), A/(\bx)) = 0$. As $\bx$ is a regular sequence we obtain $H_{d-1}( \bx, H_{m-g}(I, M)) = 0$. It follows  that\\ $\depth H_{m-g}(I, M) \geq 2$.
\section{Proof of Theorem \ref{hom}}

To prove Theorem \ref{hom} we need the following easily proved fact.
\begin{proposition}\label{d2-ex}
Let $\C$ be a bounded above chain complex (i.e., $\C_n = 0$ for $n \ll 0$). Assume $H_*(\C)  = 0$. Suppose $M$ is an $A$-module with $\Ext^i_A( \C_n, M) = 0$ for all $i > 0$
and for all $n \in \Z$. Then $H_*(\Hom_A( \C, M)) = 0$.
\end{proposition}
Next we give
\begin{proof}[Proof of Theorem \ref{hom}]
Let $\m = (u_1, \ldots, u_m)$. Let $\eta\colon  \Pb \rt \K(\bu, M)$ be  a projective resolution. Let $\C = \cone(\eta)$. We have an exact sequence of chain complexes
$$ 0 \rt \Pb \rt \C \rt \K(\bu, M)[-1] \rt 0.$$ As $\Ext^i(M. N) = 0$ for $i > 0$ we have an exact sequence of co-chain complexes
$$ 0 \rt \Hom(\K(\bu, M), N)[-1] \rt \Hom(\C, N) \rt \Hom(\Pb, N) \rt 0. $$
We note that $\C_i$ is zero for $i \ll 0$ and for all $i$ it consists of a direct sum of
a free module and some copies of $M$. So $\Ext^i_A(\C_n, N) = 0$ for all $i \geq 1$ and for all $n \in \Z$. It follows from \ref{d2-ex} that $\Hom_A(\C, N)$ is acyclic.
Therefore
$$U^i(\m, M, N) \cong H^i(\Hom(\K(\bu, M), N)) \cong H^i(\bu, \Hom_A(M, N)). $$
We also have a convergent spectral sequence
$$ \Ext_A^p(H_q(\m, M), N) \Rightarrow U^{p+q}(\m, M, N)  \cong H^{p+q}(\bu, \Hom_A(M, N)). $$
For $p + q \leq d -1$ then note that $p \leq d-1$. As $N$ is MCM $A$-module we get
$\Ext_A^p(H_q(\m, M), N) = 0$. It follows that   $H^i(\bu, \Hom_A(M, N)) = 0$ for $i \leq d-1$. So $\Hom_A(M, N)$ is a MCM $A$-module.
\end{proof}
We construct bountiful examples of modules satisfying the hypotheses of the theorem over complete intersections of co-dimension at least two.
\begin{example}
\label{const-mcm} Let $(Q, \n)$ be a complete regular local ring with algebraically closed residue field $k$. Let $\fb = f_1, \ldots, f_r \in \n^2$ be a regular sequence. Assume $r \geq 2$.
Set $A = Q/(\fb)$. Let $X$ be any non-empty projective variety in $\mathbb{P}^{r-1}(k)$ with $X \neq \mathbb{P}^{r-1}(k)$.  Choose $N$ MCM with support variety $V(N) = X$, see \cite[2.3]{B}. Let $a \in \mathbb{P}^{r-1}(k) \setminus X$ be any.
Choose $M$ MCM with  support variety $V(M) =\{a \}$. Note by \cite[5.6, 4.7]{avr-b} it follows that $\Ext^i_A(M, N) = 0$ for all $i > 0$ (similarly  $\Ext^i_A(\Syz^A_1(M), N) = 0$ for all $i > 0$).
If $\Hom_A(M, N) \neq 0$ then by the above result it follows that $\Hom_A(M, N)$ is a MCM $A$-module. If $\Hom_A(M, N) = 0$ then note that if $P$ is a minimal prime of $N$ then $M_P = 0$. Note $\Syz^A_1(M)_P \neq 0$. So it follows that $\Hom_A(\Syz^A_1(M), N) \neq 0$. Hence by our result it is a MCM $A$-module.
\end{example}
\section{Proof of Theorems \ref{gor-cm} and \ref{two-add}}
In this section we give proofs of Theorems \ref{gor-cm} and \ref{two-add}. We first give
\begin{proof}[Proof of Theorem \ref{gor-cm}]
 Let $I = (u_1, \ldots, u_m)$ (minimally).
 We take the projective resolution of $\K = \K(\bu)$ to be itself. So we get
 $U^i(\bu, A, A) = H^i(\bu, A)$. We also have a convergent spectral sequence
 \[
 E^{p,q}_2 = \Ext^{p}_A(H_q(\bu), A) \Rightarrow H^{p+q}(\bu, A).
 \]
 We note that $H_q(\bu)$ are of dimension $d - g$ and $I \subseteq \ann H_i(\bu)$ for all
 $i$. Notice $E_2^{p,q} = 0$ for all $p < g$. Furthermore for $p > g$ and $q \leq r -1$ we get $E_2^{p,q} = 0$. It follows that
 \[
 E^{g,r}_\infty = E^{g,r}_2 =  \Ext^{g}_A(H_r(\bu), A).
 \]
 We consider $E^{p,q}_2 $ when $p + q = g + r$. When $p > g$ then $q < r$. So $E_2^{p,q} = 0$. If $p <g $ then also $E_2^{p,q} = 0$ as $I \subseteq \ann H_i(\bu)$.
 It follows that
 \[
 \Ext^{g}_A(H_r(\bu), A) = U^{g + r}(\bu, A, A) = H^{g + r}(\bu, A).
 \]
 The result follows.
 \end{proof}
 Next we give
 \begin{proof}[Proof of Theorem \ref{two-add}]
 Let $I = (u_1, \ldots, u_m)$ (minimally).
 We note that \\  $V^i(\emptyset, k, A) = 0 $ for $i < d$. It follows from \ref{prop} that
 $V^i(\bu, k, A) = 0 $ for $i < d$.
As $A$ is Gorenstein note that $H_2(\bu)$ is also \CM. Note $\dim H_i(\bu) = d - g$ for $0 \leq i \leq 2$. We also have a convergent spectral sequence
\[
 E^{p,q}_2 = \Ext^{p}_A(k, H^q(\bu)) \Rightarrow V^{p+q}(\bu,k, A).
\]
We note that $ E^{p,q}_2 =  0$ for $p < d -g$ and $q = g, g+2$. Also  $ E^{p,q}_2 =  0$ for $q \notin \{ g, g+1, g+2 \}$.
It follows that for $p < d -g-2$
\[
 E^{p,g+1}_\infty =   E^{p,g + 1}_2 =  \Ext^{p}_A(k, H^{g+1}(\bu)).
\]
As $r = p + g +1 < d -1$ we get $V^r(\bu, k, A) = 0 $. So $\Ext^{p}_A(k, H^{g+1}(\bu)) = 0$ for $p < d -g-2$. The result follows.
\end{proof}
Next we give
\begin{proof}[Proof of Corollary \ref{gul}]
Suppose if possible $\injdim_{A/I} H_{l-g-1}(I) $ is finite.
By Theorem \ref{gor-cm} we have
$$H^{g+1}(I)  \cong \Ext_A^g(H_1(I), A).$$
As $H_1(I)$ is \CM \ it follows that $H^{g+1}(I)$ is \CM. Furthermore $\dim H^{g+1}(I)  = \dim A/I$. It follows that $H^{g+1}(I) = \omega_{A/I}^r$ for some $r \geq 1$ (here $ \omega_{A/I}$
denotes the canonical module of $A/I$). By duality we have
$$H_1(I) = \Ext^{g}_A(\omega_{A/I}^r, A) = (A/I)^r.$$
So $H_1(I)$ is a non-zero free $A/I$-module. This contradicts Gulliksen's result. The result follows.
\end{proof}

\end{document}